\documentclass{article}%
\usepackage{amssymb}
\usepackage{graphicx}
\usepackage{amsmath}%
\usepackage{amsfonts}
\providecommand{\U}[1]{\protect\rule{.1in}{.1in}}
\newtheorem{theorem}{Theorem}

\newtheorem{lemma}[theorem]{Lemma}

\newenvironment{proof}[1][Proof]{\textbf{#1.} }{\ \rule{0.5em}{0.5em}}
\begin{document}

\begin{center}
\bigskip Two-parameter Sturm-Liouville problems

\vspace{1cm}B. Chanane\footnote{Corresponding author} and A. Boucherif

Department of Mathematics and Statistics,

King Fahd University of Petroleum and Minerals

Dhahran 31261, Saudi Arabia.

chanane@kfupm.edu.sa

\bigskip May 4, 2012
\end{center}

\textbf{Abstract }This paper deals with the\ computation of the eigenvalues of
two-parameter Sturm-Liouville (SL)\ problems using the \textit{Regularized
Sampling Method}, a method which has been effective in computing the
eigenvalues of broad classes of SL problems (Singular, Non-Self-Adjoint,
Non-Local, Impulsive,...). We have shown, in this work that it can tackle
two-parameter SL problems with equal ease. An example was provided to
illustrate the effectiveness of the method.

\renewcommand{\theequation}{\thesection.\arabic{equation}}

\section{Introduction}

In an interesting paper published in 1963, F. M. Arscott [2] showed that the
method of separation of variables used in solving boundary value problems for
Laplace's equation leads to a two-parameter eigenvalue problem for ordinary
differential equations with the auxiliary requirement that the solutions
satisfy boundary conditions at several points. This has led to an extensive
development of multiparameter spectral theory for linear operators (see for
instance [3-4], [7-10], [12], [15], [30-33], [35], [37], [39-42]). In the
paper [12], the authors give an overview results on two-parameter eigenvalue
problems for second order linear differential equations. Several properties of
corresponding eigencurves are given.In [15] the authors have obtained
interesting geometric properties of the eigencurves (for instance transversal
intersections is equivalent to simplicity of the eigenvalues in the sense of
Chow and Hale). All the above works are concerned with the theoretical aspect
of existence of eigenvalues. Also, several authors have dealt with the
theoretical numerical analysis of two-parameter eigenvalue problems (see [5],
[11], [13], [34], [36], [38] and the references therein). Eigenvalue problems
have played a major role in the applied sciences. Consequently, the problem of
computing eigenvalues of one-parameter problems has attracted many researchers
(see for example [1, 6, 16,17,19-29] and the references therein).

Concerning the computations of eigenvalues of one-parameter Sturm-Liouville
problems, the authors in [16] introduced a new method based on Shannon's
sampling theory. It uses the analytic properties of the boundary function. The
method has been generalized to a class of singular problems of Bessel type
[17], to more general boundary conditions, separable [19] and coupled [21], to
random Sturm-Liouville problems [23] and to fourth order regular
Sturm-Liouville problems [28]. The books by Atkinson [4], Chow and Hale [30] ,
Faierman [33] , McGhee and Picard [37], Sleeman [39], Volkmer [42] and the
long awaited mongraph by Atkinson and Mingarelli \cite{AM2011} contains
several results on eigenvalues of multiparameter Sturm-Liouville problems and
the corresponding bifurcation problems. However, no attempt has been made to
compute the eigenvalues of two-parameter Sturm-Liouville problems using the
approach based on the \textit{Regularized Sampling Method} introduced recently
by the first author in \cite{c2005} to compute the eigenvalues of general
Sturm-Liouville problems and extended to the case of Singular \cite{c2007d},
Non-Self-Adjoint \cite{c2007c}, Non-Local \cite{c2008}, Impulsive SLPs
\cite{c2007b},\cite{c2007a}. We shall consider, in this paper the computation
of the eigenpairs of two-parameter Sturm-Liouville problems with three-point
boundary conditions using the \textit{Regularized Sampling Method}.

\section{The Characteristic Function}

\setcounter{equation}{0} Consider the two-parameter Sturm-Liouville problem%

\begin{equation}
\left\{
\begin{array}
[c]{c}%
-y^{\prime\prime}+qy=(\mu_{1}^{2}w_{1}+\mu_{2}^{2}w_{2})y\text{ , }0<x<1\\
y(0)=0\text{ , }y(c)=0\text{ , }y(1)=0
\end{array}
\right. \label{2.1}%
\end{equation}
where $w_{1}$, $w_{2}$ are positive and in $\mathcal{C}^{2}[0,1]$ and $q\in
L[0,1]$ and $c\in(0,1)$ some given constant.

By an eigenvalue of (\ref{2.1}) we mean a value of the couple $(\mu_{1}%
,\mu_{2})$ for which problem (\ref{2.1}) has a non trivial solution.
Conditions that insure the existence of eigenvalue are given in [3], [9],
[10], [15], [30], [31], [40] and [41]. In fact, under fairly general
conditions it has been shown (see [15], [30]) that there are smooth curves of
eigenvalues (actually eigenpairs). Our objective is to effectively localize
the eigencurves in the parameter $(\mu_{1},\mu_{2})$-plane.We should point out
that we have restricted our attention to Dirichlet boundary conditions in
order to eliminate technical details that might obscur the ideas.\bigskip

We shall associate to (\ref{2.1}) the initial value problem
\begin{equation}
\left\{
\begin{array}
[c]{c}%
y^{\prime\prime}+(\mu_{1}^{2}w_{1}+\mu_{2}^{2}w_{2})y=qy\text{ , }0<x<1\\
y(0)=0\text{ , }y^{\prime}(0)=1
\end{array}
\right. \label{2.2}%
\end{equation}
and deal first with the unperturbed case ($q=0$) then with the perturbed case
($q\neq0$).

\subsection{The unperturbed case ($q=0$)}

In this case, (\ref{2.2}) reduces to
\begin{equation}
\left\{
\begin{array}
[c]{c}%
\varphi_{1}^{\prime\prime}+w\varphi_{1}=0\text{ , }0<x<1\\
\varphi_{1}(0)=0\text{ , }\varphi_{1}^{\prime}(0)=1
\end{array}
\right. \label{2.3}%
\end{equation}
where $w=\mu_{1}^{2}w_{1}+\mu_{2}^{2}w_{2}$.

\begin{theorem}
The solution $\varphi_{1}$ of (\ref{2.3}) is an entire function of $(\mu
_{1},\mu_{2})\in\mathbb{C}^{2}$ for each fixed $x\in(0,1]$ of order $(1,1) $
and type $(\sigma_{1}(x),\sigma_{2}(x))$ and satisfies the estimate,%
\[
\left\vert \varphi_{1}(x)\right\vert \leq K_{1}\exp\left[  \sigma
_{1}(x)\left\vert \mu_{1}\right\vert +\sigma_{2}(x)\left\vert \mu
_{2}\right\vert \right]  \text{ , }(\mu_{1},\mu_{2})\in\mathbb{C}^{2}%
\]
for each fixed $x\in(0,1]$ \ where, $\sigma_{i}(x)=2\left\{  x\int_{0}%
^{x}w_{i}(\xi)d\xi\right\}  ^{\frac{1}{2}}$for $i=1,2$.
\end{theorem}

\begin{proof}
From (\ref{2.3}) we get the integral equation
\begin{equation}
\varphi_{1}(x)=x-\int_{0}^{x}(x-\xi)\left[  \mu_{1}^{2}w_{1}(\xi)+\mu_{2}%
^{2}w_{2}(\xi)\right]  \varphi_{1}(\xi)d\xi\label{2.9}%
\end{equation}

Let
\begin{equation}
\left\{
\begin{array}
[c]{c}%
\varphi_{1,0}(x)=x\\
\varphi_{1,n+1}(x)=\int_{0}^{x}(x-\xi)\left[  \mu_{1}^{2}w_{1}(\xi)+\mu
_{2}^{2}w_{2}(\xi)\right]  \varphi_{1,n}(\xi)d\xi\text{ , }n\geq0
\end{array}
\right. \label{2.10}%
\end{equation}
We shall show, by induction on $n$, that
\begin{equation}
\left\vert \varphi_{1,n}(x)\right\vert \leq\frac{x}{n!(n+1)!}\left\{
x\int_{0}^{x}\left[  \left\vert \mu_{1}\right\vert ^{2}w_{1}(\xi)+\left\vert
\mu_{2}\right\vert ^{2}w_{2}(\xi)\right]  d\xi\right\}  ^{n}\text{ , }%
n\geq0\label{2.11}%
\end{equation}
It is true for $n=0$. Assume it is true for $n$. We shall show that it is true
for $n+1$. Indeed, from (\ref{2.10}), we have
\begin{align}
\left\vert \varphi_{1,n+1}(x)\right\vert  & \leq\int_{0}^{x}(x-\xi)\left[
\left\vert \mu_{1}\right\vert ^{2}w_{1}(\xi)+\left\vert \mu_{2}\right\vert
^{2}w_{2}(\xi)\right]  \times\nonumber\\
& \frac{\xi}{n!(n+1)!}\left\{  \xi\int_{0}^{\xi}\left[  \left\vert \mu
_{1}\right\vert ^{2}w_{1}(\tau)+\left\vert \mu_{2}\right\vert ^{2}w_{2}%
(\tau)\right]  d\tau\right\}  ^{n}d\xi\text{ }\label{2.12}%
\end{align}
Using the fact that the expression $(x-\xi)\xi^{n+1}$ attains its maximum at
$\xi=\frac{n+1}{n+2}x$, we get
\begin{align}
\left\vert \varphi_{1,n+1}(x)\right\vert  & \leq\frac{x}{n!(n+1)!}\frac
{1}{n+1}\frac{1}{n+2}\left(  \frac{n+1}{n+2}x\right)  ^{n+1}\left\{  \int
_{0}^{x}\left[  \left\vert \mu_{1}\right\vert ^{2}w_{1}(\tau)+\left\vert
\mu_{2}\right\vert ^{2}w_{2}(\tau)\right]  d\tau\right\}  ^{n+1}\nonumber\\
& \leq\frac{x}{(n+1)!(n+2)!}\left\{  x\int_{0}^{x}\left[  \left\vert \mu
_{1}\right\vert ^{2}w_{1}(\tau)+\left\vert \mu_{2}\right\vert ^{2}w_{2}%
(\tau)\right]  d\tau\right\}  ^{n+1}\label{2.13}%
\end{align}
that is (\ref{2.11}) is true for $n+1$. Hence, it is true for all $n\geq0$.

Now, $\varphi_{1}(x)=\sum_{n\geq0}(-1)^{n}\varphi_{1,n}(x)$ and the series is
absolutely and uniformly convergent since
\begin{align}
\left\vert \varphi_{1}(x)\right\vert  & =\left\vert \sum_{n\geq0}%
(-1)^{n}\varphi_{1,n}(x)\right\vert \leq\sum_{n\geq0}\left\vert \varphi
_{1,n}(x)\right\vert \nonumber\\
& \leq x\sum_{n\geq0}\frac{1}{n!(n+1)!}\left\{  x\int_{0}^{x}\left[
\left\vert \mu_{1}\right\vert ^{2}w_{1}(\tau)+\left\vert \mu_{2}\right\vert
^{2}w_{2}(\tau)\right]  d\tau\right\}  ^{n}\nonumber\\
& =xI_{1}(\left\{  2\sqrt{x\int_{0}^{x}\left[  \left\vert \mu_{1}\right\vert
^{2}w_{1}(\tau)+\left\vert \mu_{2}\right\vert ^{2}w_{2}(\tau)\right]  d\tau
}\right\}  )\times\nonumber\\
& \left\{  x\int_{0}^{x}\left[  \left\vert \mu_{1}\right\vert ^{2}w_{1}%
(\tau)+\left\vert \mu_{2}\right\vert ^{2}w_{2}(\tau)\right]  d\tau\right\}
^{-\frac{1}{2}}\label{2.14}%
\end{align}
where $I_{1}$ is the modified Bessel function of the first kind order 1
\[
I_{1}(z)=\sum_{n\geq0}\frac{1}{n!(n+1)!}\left(  \frac{z}{2}\right)  ^{2n+1}.
\]
Using the fact that $I_{1}(z)\sim\frac{e^{z}}{\sqrt{2\pi z}}$ as
$z\rightarrow\infty,$ we get
\begin{align}
\left\vert \varphi_{1}(x)\right\vert  & \leq K_{1}\exp\left[  2\left\{
x\int_{0}^{x}\left[  \left\vert \mu_{1}\right\vert ^{2}w_{1}(\tau)+\left\vert
\mu_{2}\right\vert ^{2}w_{2}(\tau)\right]  d\tau\right\}  ^{\frac{1}{2}%
}\right] \nonumber\\
& \leq K_{1}\exp\left[  \sigma_{1}(x)\left\vert \mu_{1}\right\vert +\sigma
_{2}(x)\left\vert \mu_{2}\right\vert \right] \label{2.15}%
\end{align}
Therefore, $\varphi_{1}$ is an entire function of $(\mu_{1},\mu_{2}%
)\in\mathbb{C}^{2}$, as a uniformly convergent series of entire functions, for
each fixed $x\in(0,1],$ of order $(1,1)$ and type $(\sigma_{1}(x),\sigma
_{2}(x)).$This concludes the proof.\bigskip
\end{proof}

\bigskip We shall make use of the Liouville-Green's transformation
\begin{equation}
\left\{
\begin{array}
[c]{c}%
t(x)=\int_{0}^{x}\sqrt{w(\xi)}d\xi\\
z(t)=\left\{  w(x)\right\}  ^{\frac{1}{4}}\varphi_{1}(x)
\end{array}
\right. \label{2.25}%
\end{equation}
to bring (\ref{2.2}) to the form
\begin{equation}
\left\{
\begin{array}
[c]{c}%
\frac{d^{2}z}{dt^{2}}+\left\{  1+r(t)\right\}  z=0\text{ , }0<t<\int_{0}%
^{1}\sqrt{w(\xi)}d\xi\\
z(0)=0\text{ , }\frac{dz}{dt}(0)=\left\{  w(0)\right\}  ^{-\frac{1}{4}}%
\end{array}
\right. \label{2.5}%
\end{equation}
which can be written as an integral equation
\begin{equation}
z(t)=\left\{  w(0)\right\}  ^{-\frac{1}{4}}\sin t-\int_{0}^{t}\sin
(t-\tau)r(\tau)z(\tau)d\tau\label{2.6}%
\end{equation}
where $r(t)=\left[  \left\{  w(x)\right\}  ^{-\frac{3}{4}}\frac{d^{2}}{dx^{2}%
}\left\{  w(x)\right\}  ^{-\frac{1}{4}}\right]  _{|x=x(t)}$.

Returning to the original variables, we deduce that $\varphi_{1}$ satisfies
the integral equation
\begin{equation}
\varphi_{1}(x)=\left\{  w(0)w(x)\right\}  ^{-\frac{1}{4}}\sin\left\{  \int
_{0}^{x}\sqrt{w(\xi)}d\xi\right\}  -\int_{0}^{x}\sin(\int_{\overline{x}}%
^{x}\sqrt{w(\xi)}d\xi)\psi(x,\overline{x})\varphi_{1}(\overline{x}%
)d\overline{x}\label{2.7}%
\end{equation}
where
\begin{equation}
\psi(x,\overline{x})=\left\{  w(x)\right\}  ^{-\frac{1}{4}}\left\{
w(\overline{x})\right\}  ^{-\frac{3}{4}}\left[  -\frac{1}{4}w^{\prime\prime
}(\overline{x})\left\{  w(\overline{x})\right\}  ^{-\frac{1}{2}}+\frac{5}%
{16}\left\{  w^{\prime}(\overline{x})\right\}  ^{2}\left\{  w(\overline
{x})\right\}  ^{-\frac{3}{2}}\right] \label{2.8}%
\end{equation}

We shall present next some estimates whose proofs are immediate and left to
the reader.

\begin{lemma}
The function $\psi(x,\overline{x})$ satisfies the estimate
\begin{equation}
\left\vert \psi(x,\overline{x})\right\vert \sim\frac{K_{2}}{\left\vert \mu
_{1}\right\vert +\left\vert \mu_{2}\right\vert }\text{, as }\left\vert \mu
_{1}\right\vert +\left\vert \mu_{2}\right\vert \rightarrow\infty,(\mu_{1}%
,\mu_{2})\in\mathbb{R}^{2}\text{ }\label{2.16}%
\end{equation}

\end{lemma}

\begin{lemma}
The function $\alpha$ defined by
\[
\alpha(x)=\left(  \text{sinc}\left\{  \sigma_{1}(x)\mu_{1}+\sigma_{2}%
(x)\mu_{2}\right\}  \right)  ^{m}\text{ }%
\]
where $sinc(z)=z^{-1}\sin z$ and $m$ is a positive integer, is an entire
function of $(\mu_{1},\mu_{2})\in\mathbb{C}^{2}$ for each fixed $x\in(0,1]$ of
order $(1,1)$ and type $(\sigma_{1}(x),\sigma_{2}(x))$. Furthermore, $\alpha$
satisfies the estimate \
\[
\left\vert \alpha(1)\right\vert \sim\frac{K_{3}}{\left\vert \mu_{1}\right\vert
+\left\vert \mu_{2}\right\vert }\text{, as }\left\vert \mu_{1}\right\vert
+\left\vert \mu_{2}\right\vert \rightarrow\infty,(\mu_{1},\mu_{2}%
)\in\mathbb{R}^{2}.
\]

\end{lemma}

\begin{lemma}
The function $\varphi_{1}$ satisfies the estimate
\begin{equation}
\left\vert \varphi_{1}(1)\right\vert \sim\frac{K_{4}}{\left\vert \mu
_{1}\right\vert +\left\vert \mu_{2}\right\vert }\text{, as }\left\vert \mu
_{1}\right\vert +\left\vert \mu_{2}\right\vert \rightarrow\infty,(\mu_{1}%
,\mu_{2})\in\mathbb{R}^{2}\text{ }\label{2.17}%
\end{equation}

\end{lemma}

\bigskip Combining the above results, we obtain the following theorem,

\begin{theorem}
The function $\alpha\varphi_{1}$ is an entire function of $(\mu_{1},\mu
_{2})\in\mathbb{C}^{2}$ for each fixed $x\in(0,1]$ of order $(1,1)$ and type
$((m+1)\sigma_{1}(x),(m+1)\sigma_{2}(x))$ and satisfies the estimate \
\[
\left\vert \alpha(x)\varphi_{1}(x)\right\vert \sim\frac{K(x)}{\left(
\left\vert \mu_{1}\right\vert +\left\vert \mu_{2}\right\vert \right)  ^{m+1}%
}\text{, as }\left\vert \mu_{1}\right\vert +\left\vert \mu_{2}\right\vert
\rightarrow\infty,(\mu_{1},\mu_{2})\in\mathbb{R}^{2}.
\]

\end{theorem}

where $K$ depends on $x\in(0,1]$ but is independent of $(\mu_{1},\mu_{2})$ .

Let $PW_{\beta_{1},\beta_{2}}$ denote the Paley-Wiener space,%
\[
PW_{\beta_{1},\beta_{2}}=\left\{
\begin{array}
[c]{c}%
h(z_{1},z_{2})\text{ entire / }\left\vert h(z_{1},z_{2})\right\vert \leq
C\exp\left\{  \beta_{1}\left\vert z_{1}\right\vert +\beta_{2}\left\vert
z_{2}\right\vert \right\}  \text{,}\\
\text{ }\int_{-\infty}^{\infty}\int_{-\infty}^{\infty}\text{ }\left\vert
h(z_{1},z_{2})\right\vert ^{2}dz_{1}dz_{2}<\infty
\end{array}
\right\}
\]
we have,

\begin{theorem}
$\alpha(x)\varphi_{1}(x)$ , as a function of $(\mu_{1},\mu_{2})$ belongs to
the Paley-Wiener space $PW_{\beta_{1},\beta_{2}}$ where $\left(  \beta
_{1},\beta_{2}\right)  =((m+1)\sigma_{1}(x),(m+1)\sigma_{2}(x))$ for each
fixed $x\in(0,1].$
\end{theorem}

\bigskip

\subsection{The perturbed case ($q\neq0$)}

\bigskip Let $\varphi_{1}$ , $\varphi_{2}$ be two linearly independent
solutions of $\varphi^{\prime\prime}+w\varphi=0$ satisfying $\varphi
_{1}(0)=\varphi_{2}^{\prime}(0)=0$ , $\varphi_{1}^{\prime}(0)=\varphi
_{2}(0)=1$ then the method of variation of parameters shows that (\ref{2.2})
can be written as the integral equation
\begin{equation}
y(x)=\varphi_{1}(x)+\int_{0}^{x}\Phi(x,\xi)q(\xi)y(\xi)d\xi\label{2.19}%
\end{equation}
where $\Phi(x,\xi)=\varphi_{1}(\xi)\varphi_{2}(x)-\varphi_{2}(\xi)\varphi
_{1}(x)$.

\bigskip Here again, it is not hard to show that $y(x)$ is an entire function
of $(\mu_{1},\mu_{2})$ for each $x\in(0,1]$, of order $(1,1)$ and type
$(\sigma_{1}(x),\sigma_{2}(x))$. Multiplication by $\alpha$ gives a function
$\alpha(x)y(x)$ of $(\mu_{1},\mu_{2})$ in a Paley-Wiener space $PW_{\beta
_{1},\beta_{2}}$ for each $x\in(0,1]$. More specifically, we have the following,

\begin{theorem}
The function $\widetilde{y}$ defined by $\widetilde{y}(x)=\alpha(x)y(x)$ ,
belongs to $PW_{\beta_{1},\beta_{2}}$ where $\left(  \beta_{1},\beta
_{2}\right)  =((m+1)\sigma_{1}(x),(m+1)\sigma_{2}(x))$ as a function of
$(\mu_{1},\mu_{2})\in\mathbb{C}$ for each $x\in(0,1]$, and satisfies the
estimate,%
\[
\left\vert \alpha(x)y(x)\right\vert \sim\frac{K(x)}{\left(  \left\vert \mu
_{1}\right\vert +\left\vert \mu_{2}\right\vert \right)  ^{m+1}},\text{as
}\left\vert \mu_{1}\right\vert +\left\vert \mu_{2}\right\vert \rightarrow
\infty,(\mu_{1},\mu_{2})\in\mathbb{R}^{2}.
\]
where $K$ depends on $x\in(0,1]$ but is independent of $(\mu_{1},\mu_{2})$ .
\end{theorem}

\bigskip

\begin{proof}
Since $\Phi_{xx}+w\Phi=0$ , $\Phi(t,t)=0$ and $\Phi_{x}(t,t)=1$, we have,
\begin{align}
\Phi(x,t)  & =\left\{  w(t)w(x)\right\}  ^{-\frac{1}{4}}\sin\left(  \int
_{t}^{x}\sqrt{w(\xi)}d\xi\right) \nonumber\\
& -\int_{t}^{x}\sin\left(  \int_{\overline{x}}^{x}\sqrt{w(\xi)}d\xi\right)
\psi(x,\overline{x})\Phi(\overline{x},t)d\overline{x}\label{2.20}%
\end{align}
\newline so that,
\begin{equation}
\left\vert \Phi(x,t)\right\vert \sim\frac{K_{4}}{\left\vert \mu_{1}\right\vert
+\left\vert \mu_{2}\right\vert }\leq K_{5}\text{, as }\left\vert \mu
_{1}\right\vert +\left\vert \mu_{2}\right\vert \rightarrow\infty,(\mu_{1}%
,\mu_{2})\in\mathbb{R}^{2}.\label{2.22}%
\end{equation}
from which we get, after using Gronwall's lemma on (\ref{2.19})and the
estimate for $\ \varphi_{1},$
\[
\left\vert y(x)\right\vert \leq K_{1}\exp\left[  \sigma_{1}(x)\left\vert
\mu_{1}\right\vert +\sigma_{2}(x)\left\vert \mu_{2}\right\vert \right]
\exp\left\{  K_{5}\int_{0}^{x}|q(t)|dt\right\}
\]%
\begin{equation}
\left\vert y(1)\right\vert \leq K_{7}\exp\left[  \sigma_{1}(1)\left\vert
\mu_{1}\right\vert +\sigma_{2}(1)\left\vert \mu_{2}\right\vert \right]
\text{, as }\left\vert \mu_{1}\right\vert +\left\vert \mu_{2}\right\vert
\rightarrow\infty,(\mu_{1},\mu_{2})\in\mathbb{R}^{2}.\label{2.26}%
\end{equation}
and
\begin{equation}
\left\vert \alpha(1)y(1)\right\vert \leq K_{8}\exp\left[  (m+1)\sigma
_{1}(1)\left\vert \mu_{1}\right\vert +(m+1)\sigma_{2}(1)\left\vert \mu
_{2}\right\vert \right]  \text{, as }\left\vert \mu_{1}\right\vert +\left\vert
\mu_{2}\right\vert \rightarrow\infty,(\mu_{1},\mu_{2})\in\mathbb{R}%
^{2}.\label{2.28}%
\end{equation}
Furthermore, we have,%
\[
\left\vert \alpha(x)y(x)\right\vert \sim\frac{K(x)}{\left(  \left\vert \mu
_{1}\right\vert +\left\vert \mu_{2}\right\vert \right)  ^{m+1}},\text{as
}\left\vert \mu_{1}\right\vert +\left\vert \mu_{2}\right\vert \rightarrow
\infty,(\mu_{1},\mu_{2})\in\mathbb{R}^{2}.
\]
where $K$ depends on $x\in(0,1]$ but is independent of $(\mu_{1},\mu_{2})$ .
\end{proof}

\bigskip

To summarize, in both cases, unperturbed and perturbed, the transform
$\widetilde{y}(x;\mu_{1},\mu_{2})$ of the solution $y(x;\mu_{1},\mu_{2})$ of
(\ref{2.2}) is in a Paley-Wiener space $PW_{\beta_{1},\beta_{2}}$ where
$\left(  \beta_{1},\beta_{2}\right)  =((m+1)\sigma_{1}(x),(m+1)\sigma
_{2}(x)).$ Thus $\widetilde{y}(x;\mu_{1},\mu_{2})$ can be recovered at each
$x\in(0,1]$ from its samples at the latice points $(\mu_{1j},\mu_{2k}%
)=(j\frac{\pi}{(m+1)\sigma_{1}(x)},k\frac{\pi}{(m+1)\sigma_{2}(x)})$,
$(j,k)\in\mathbb{Z}^{2}$ using the rectangular cardinal series
(\cite{50,43,44}),

\begin{theorem}
Let $f\in PW_{\beta_{1},\beta_{2}}$ then
\[
f(\mu_{1},\mu_{2})=\sum_{j=-\infty}^{\infty}\sum_{k=-\infty}^{\infty}%
f(\mu_{1j},\mu_{2k})\frac{\sin\beta_{1}(\mu_{1}-\mu_{1j})}{\beta_{1}(\mu
_{1}-\mu_{1j})}\frac{\sin\beta_{2}(\mu_{2}-\mu_{2k})}{\beta_{2}(\mu_{2}%
-\mu_{2k})}%
\]
the convergence of the series being uniform and in $L_{d\mu_{1}d\mu_{2}}%
^{2}(\mathbb{R}^{2})$, and $\mu_{mn}=n\pi/\beta_{m}$ , $m=1,2$, $n\in
\mathbb{Z}$.
\end{theorem}

\bigskip Let $\sigma_{11}=\sigma_{1}(1),$ $\sigma_{21}=\sigma_{2}(1),$
$\sigma_{12}=\sigma_{1}(c),$ $\sigma_{22}=\sigma_{2}(c).$

The eigenpairs are therefore \bigskip$(\mu_{1}^{2},\mu_{2}^{2})$ where
$(\mu_{1},\mu_{2})$ solve the nonlinear system%
\begin{equation}
\left\{
\begin{array}
[c]{c}%
B_{1}(\mu_{1},\mu_{2})=0\\
B_{2}(\mu_{1},\mu_{2})=0
\end{array}
\right. \label{2.30}%
\end{equation}
where,%
\begin{equation}
\left\{
\begin{array}
[c]{c}%
B_{1}(\mu_{1},\mu_{2})\triangleq\frac{1}{\alpha(1)}\sum_{j=-\infty}^{\infty
}\sum_{k=-\infty}^{\infty}\widetilde{y}(1;\mu_{1j},\mu_{2k})\frac
{\sin2(m+1)\sigma_{11}(\mu_{1}-\mu_{1j})}{2(m+1)\sigma_{11}(\mu_{1}-\mu_{1j}%
)}\frac{\sin2(m+1)\sigma_{21}(\mu_{2}-\mu_{2k})}{2(m+1)\sigma_{21}(\mu_{2}%
-\mu_{2k})}\\
B_{1}(\mu_{1},\mu_{2})\triangleq\frac{1}{\alpha(c)}\sum_{j=-\infty}^{\infty
}\sum_{k=-\infty}^{\infty}\widetilde{y}(c;\mu_{1j},\mu_{2k})\frac
{\sin2(m+1)\sigma_{12}(\mu_{1}-\mu_{1j})}{2(m+1)\sigma_{12}(\mu_{1}-\mu_{1j}%
)}\frac{\sin2(m+1)\sigma_{22}(\mu_{2}-\mu_{2k})}{2(m+1)\sigma_{22}(\mu_{2}%
-\mu_{2k})}%
\end{array}
\right. \label{2.27}%
\end{equation}

\section{A numerical example\bigskip}

\setcounter{equation}{0} \bigskip We shall consider in this section the
two-parameter Sturm-Liouville problem with three- point boundary conditions
given by
\begin{equation}
\left\{
\begin{array}
[c]{c}%
-y^{\prime\prime}=(\mu_{1}^{2}+\mu_{2}^{2}x)y\text{ , }0<x<1\\
y(0)=y(0.7)=y(1)
\end{array}
\right. \label{3.10}%
\end{equation}

The general solution $y$ of the first differential equation can be expressed
in terms of Ai and Bi functions and their first dirivatives as%
\begin{equation}
y(x;\mu_{1},\mu_{2})=\frac{(-1)^{2/3}\left(  \text{Ai}\left(  \frac
{\sqrt[3]{-1}\mu_{1}^{2}}{\mu_{2}^{4/3}}\right)  \text{Bi}\left(
\frac{\sqrt[3]{-1}\left(  \mu_{1}^{2}+\mu_{2}^{2}x\right)  }{\mu_{2}^{4/3}%
}\right)  -\text{Ai}\left(  \frac{\sqrt[3]{-1}\left(  \mu_{1}^{2}+\mu_{2}%
^{2}x\right)  }{\mu_{2}^{4/3}}\right)  \text{Bi}\left(  \frac{\sqrt[3]{-1}%
\mu_{1}^{2}}{\mu_{2}^{4/3}}\right)  \right)  }{\mu_{2}^{2/3}\left(
\text{Ai}^{\prime}\left(  \frac{\sqrt[3]{-1}\mu_{1}^{2}}{\mu_{2}^{4/3}%
}\right)  \text{Bi}\left(  \frac{\sqrt[3]{-1}\mu_{1}^{2}}{\mu_{2}^{4/3}%
}\right)  -\text{Ai}\left(  \frac{\sqrt[3]{-1}\mu_{1}^{2}}{\mu_{2}^{4/3}%
}\right)  \text{Bi}^{\prime}\left(  \frac{\sqrt[3]{-1}\mu_{1}^{2}}{\mu
_{2}^{4/3}}\right)  \right)  }\label{3.11}%
\end{equation}

Thus the eigenpairs $(\mu_{1}^{2},\mu_{2}^{2})$~can be obtained from the
solutions $(\mu_{1},\mu_{2})$~of the system
\begin{equation}
\left\{
\begin{array}
[c]{c}%
\widetilde{y}(1;\mu_{1},\mu_{2})=0\\
\widetilde{y}(c;\mu_{1},\mu_{2})=0
\end{array}
\right. \label{3.4}%
\end{equation}

For numerical purposes we have truncated the associated series to $|j|,|k|\leq
N=50$ and took $m=5$ in the function $\alpha$. Thus, the approximate
eigenpairs are seen as solutions of the system
\begin{equation}
\left\{
\begin{array}
[c]{c}%
\widetilde{y}^{[N]}(1;\mu_{1},\mu_{2})=0\\
\widetilde{y}^{[N\}}(c;\mu_{1},\mu_{2})=0
\end{array}
\right. \label{3.5}%
\end{equation}
The next table shows the exact eigenpairs together with their approximations
using the \textit{Regularized Sampling Method} (RSM).

\begin{center}%
\begin{tabular}
[c]{|c|c|c|c|}\hline
$\mu_{1}$ (exact) & $\mu_{2}$ (exact) & $\mu_{1}$ (RSM) & $\mu_{2}$
(RSM)\\\hline
$7.788149097670813$ & $7.5908224365786845$ & $7.78814883048352$ &
$7.590823605399021$\\\hline
$4.194056047341936$ & $22.273796542861913$ & $4.194057357011064$ &
$22.273795699572364$\\\hline
$15.597607295800684$ & $15.165889997583099$ & $15.597605904429917$ &
$15.1658929384375$\\\hline
$13.87615754699624$ & $30.597837626955204$ & $13.876157718920586$ &
$30.59783713797321$\\\hline
$23.40242244972004$ & $22.744341450749697$ & $23.40242640980619$ &
$22.744329536815076$\\\hline
$9.51130982713563$ & $44.296491367611544$ & $9.511333341150682$ &
$44.29647446214639$\\\hline
$39.31833491621106$ & $15.670248767892955$ & $39.31833152483279$ &
$15.670257662031721$\\\hline
$47.18632817417246$ & $24.620489734469363$ & $47.186345272524804$ &
$24.62038504864854$\\\hline
$46.81208951678993$ & $45.483251580157955$ & $46.813097550330646$ &
$45.480149444567346$\\\hline
$30.03437464767369$ & $46.558148101315844$ & $30.034392317076748$ &
$46.55811263809656$\\\hline
\end{tabular}

\end{center}

\bigskip

\bigskip

\section{Conclusion}

In this paper, we have successfully computed the eigenpairs of two-parameter
Sturm-Liouville problems using the regularized sampling method. A method which
has been very efficient in computing the eigenvalues of broad classes of
Sturm-Liouville problems (Singular, Non-Self-Adjoint, Non-Local,
Impulsive,...). We have shown, in this work that it can tackle two-parameter
SL problems with equal ease. An example was provided to illustrate the
effectiveness of the method.

\bigskip

\textbf{Acknowledgements }The authors are grateful to KFUPM for its usual support.

\end{document}